\documentclass[10pt]{amsart}
\usepackage{enumerate,amsmath,amssymb,latexsym,
amsfonts, amsthm, amscd, MnSymbol}

\theoremstyle{plain}

\newtheorem{theorem}{Theorem}

\numberwithin{equation}{section}

\addtocounter{section}{-1}

\begin{document}

\title {Implicational Completeness}

\date{}

\author[P.L. Robinson]{P.L. Robinson}

\address{Department of Mathematics \\ University of Florida \\ Gainesville FL 32611  USA }

\email[]{paulr@ufl.edu}

 \keywords{}

\begin{abstract}

For the implicational propositional calculus, we present a proof of completeness based on a variant of the Lindenbaum procedure. 

\end{abstract}

\maketitle

\medbreak

\section{Introduction} 

A standard formulation of the Implicational Propositional Calculus (IPC) has $\supset$ as its only connective, has modus ponens as its only inference rule, and has the following axiom schemes: 
$$A \supset (B \supset A)$$
$$[A \supset (B \supset C)] \supset [(A \supset B) \supset (A \supset C)]$$
$$[(A \supset B) \supset A] \supset A$$
the last of which is due to Peirce. A (Boolean) {\it valuation} is a map $v$ from the set ${\it wf}$ of all well-formed formulas to the set $\{ 0, 1 \}$ such that $v(A \supset B) = 0$ precisely when $v(A) = 1$ and $v(B) = 0$; a {\it tautology} is a well-formed formula that takes the value $1$ in all such valuations. 

\medbreak 

 The {\it Completeness Theorem} for IPC asserts that every tautology is a theorem: if a well-formed formula has value $1$ in each valuation then there is a proof of it using modus ponens and the three axiom schemes listed above. One proof of completeness for IPC is indicated in Exercises 6.3, 6.4, 6.5 of Robbin [2]; that proof adapts the Kalm\`ar approach for the classical propositional calculus. Our purpose in this brief paper is to present a proof of completeness for IPC that adapts the Lindenbaum approach for the classical propositional calculus. 

\medbreak 

\section{Theorem and Proof}

\medbreak 

We begin with some simple observations regarding IPC. To say that $A \in {\it wf}$ may be deduced from $\Gamma \subseteq {\it wf}$ we may write $\Gamma \vdash A$ as usual; in particular (taking $\Gamma$ to be empty) $\vdash A$ asserts that $A$ is a theorem. Modus ponens and the first two axiom schemes together ensure that if $A \in {\it wf}$ is any well-formed formula then $A \supset A$ is a theorem. Modus ponens and the first two axiom schemes further ensure that the {\it Deduction Theorem} (DT) holds: if $\Gamma \cup \{ A \} \vdash B$ then $\Gamma \vdash A \supset B$; in particular, if $A \vdash B$ then $\vdash A \supset B$. 
\medbreak 

The lack of negation in IPC is in part repaired by an elegant device. Fix an arbitrary well-formed formula $Q \in {\it wf}$. When $A \in {\it wf}$ write $Q A := Q(A) := A \supset Q$. 

\medbreak 

\begin{theorem} \label{Robbin}
Fix any well-formed formula $Q \in {\it wf}$. If $A, B, C \in {\it wf}$ then each of the following well-formed formulas is a theorem of IPC: \par 

{\rm (1)} $(A \supset B) \supset [(B \supset C) \supset (A \supset C)]$ \par 
{\rm (2)} $(A \supset B) \supset (QB \supset QA)$ \par 
{\rm (3)} $A \supset QQA$ \par
{\rm (4)} $QQQA \supset QA$ \par 
{\rm (5)} $QQB \supset QQ(A \supset B)$ \par 
{\rm (6)} $QQA \supset [QB \supset Q(A \supset B)]$ \par 
{\rm (7)} $QA \supset QQ(A \supset B)$ \par 
{\rm (8)} $(QA \supset B) \supset [(QQA \supset B) \supset QQB].$ 

\end{theorem}

\begin{proof} 
This is Exercise 6.3 in Chapter 1 of Robbin [2]. As noted by Robbin, part (7) requires the Peirce axiom scheme; the other parts need only the first two axiom schemes. 
\end{proof} 

\medbreak 

 The classical propositional calculus presented by Church [1] and followed by Robbin [2] incorporates a propositional symbol $\mathfrak{f}$ (falsity) having value $0$ under each valuation; it takes $\mathfrak{f} A = A \supset \mathfrak{f}$ for the negation $\thicksim A$ of $A$; and it replaces the Peirce axiom scheme by the `double negation' axiom scheme  $\thicksim \thicksim A \supset A$. The symbol $\mathfrak{f}$ has no place in the present paper; however, significant aspects of its function will be served by a non-theorem $Q$.

\medbreak 

From this point on, we shall consider extensions of IPC obtained by enlarging the set of theorems, so we shall modify our notation accordingly. Let us write $L$ for the system IPC as formulated above and write $\mathbb{T} (L) = \{ A \in {\it wf} : \; \; \vdash A \}$  for its set of theorems. An extension $M$ of $L$ is produced by adding an axiom (or axioms); thus $\mathbb{T} (L) \subseteq \mathbb{T} (M)$ and the Deduction Theorem continues to hold for $M$. To indicate that a well-formed formula $A$ is a theorem of $M$ we prefer to write $A \in\mathbb{T} (M)$ rather than the customary $\underset{M}{\vdash} A$. 

\medbreak 

\begin{theorem} \label{con}
Let $Q \in {\it wf}$ and let $M$ be an extension of $L$. The following are equivalent: \par 
{\rm (1)} $Q \in \mathbb{T} (M);$ \par 
{\rm (2)}  some $A \in {\it wf}$ has $A \in \mathbb{T} (M)$ and $QA \in \mathbb{T} (M);$ \par 
{\rm (3)} some $A \in {\it wf}$ has $QQA \in \mathbb{T} (M)$ and $QA \in \mathbb{T} (M).$
\end{theorem} 

\begin{proof} 
$(2) \Rightarrow (3)$ Follows from Theorem \ref{Robbin} part (3) by modus ponens. \\
$(3) \Rightarrow (1)$ Follows by modus ponens from $\mathbb{T} (M) \ni QA$ and $\mathbb{T} (M) \ni QQA \; (= QA \supset Q)$. \\
$(1) \Rightarrow (2)$ Simply let $A = Q$ and recall that $QQ \; (= Q \supset Q) \in  \mathbb{T} (L) \subseteq \mathbb{T} (M).$
\end{proof} 

\medbreak 

We say that $M$ is $Q$-{\it inconsistent} precisely when it satisfies one (hence each) of the equivalent conditions in this theorem; we say that $M$ is $Q$-{\it consistent} otherwise. 

\medbreak 

Henceforth, we shall let $Q \in {\it wf}$ be a well-formed formula that is {\bf not} a theorem of $L$; thus, $L$ is $Q$-consistent. 

\medbreak 

\begin{theorem} \label{step}
Let $M$ be a $Q$-consistent extension of $L$ and $A \in {\it wf}$  a well-formed formula. If $QQA$ is {\it not} a theorem of $M$, then the extension $N$ of $M$ obtained by adding $QA$ as an axiom is $Q$-consistent.  
\end{theorem} 

\begin{proof} 
To prove the contrapositive, assume that $Q \in \mathbb{T} (N)$: a deduction of $Q$ within the system $N$ is a deduction of $Q$ from $QA$ within the system $M$; by the Deduction Theorem, it follows that $QQA = (QA \supset Q) \in \mathbb{T} (M)$. 
\end{proof}

\medbreak 

We say that $M$ is $Q$-{\it complete} precisely when each $A \in {\it wf}$ satisfies either $QQA \in \mathbb{T} (M)$ or $QA \in \mathbb{T} (M)$; that is, either $QQA$ or $QA$ is a theorem of $M$. 

\medbreak 

\begin{theorem} \label{completion}
Each $Q$-consistent extension $M$ of $L$ has a $Q$-complete $Q$-consistent extension. 
\end{theorem} 

\begin{proof} 
List all the well-formed formulas: say ${\it wf} = \{ A_n : n \geqslant 0 \} = \{ A_0, A_1, \dots \}.$ \par 
Put $N_0 = M$. If $QQA_0 \in \mathbb{T} (N_0)$ then let $N_1 = N_0$; if $QQA_0 \notin \mathbb{T} (N_0)$ then let $N_1$ be $N_0$ with $Q A_0$ as an extra axiom. Repeat inductively: if $QQA_n \in \mathbb{T} (N_n)$ then let $N_{n + 1} = N_n$; if $QQA_n \notin \mathbb{T} (N_n)$ then let $N_{n + 1}$ be $N_n$ with $Q A_n$ as an extra axiom. Finally, let $N$ be the extension of $M$ produced by adding as axioms all those {\it wf}s introduced at each stage of this inductive process. \par 
Claim: $N$ is $Q$-consistent. [Any proof of $Q$ in $N$ would involve only finitely many axioms and would therefore be a proof of $Q$ in $N_n$ for some $n \geqslant 0$; but Theorem \ref{step} guarantees inductively that the extension $N_n$ is $Q$-consistent for each $n \geqslant 0$.] \par 
Claim: $N$ is $Q$-complete. [Take any {\it wf}: say $A_n$. If $QQA_n \in \mathbb{T} (N_n)$ then $QQA_n \in \mathbb{T} (N)$; if $QQA_n \notin \mathbb{T} (N_n)$ then $QA_n \in \mathbb{T} (N_{n + 1})$ so that $QA_n \in \mathbb{T} (N)$. Thus, if $A \in {\it wf}$ is arbitrary then either $QQA$ or $QA$ is a theorem of $N$ as required.] 
\end{proof} 

\medbreak 

Now, let $M$ be a $Q$-consistent extension of $L$ and let $N$ be a $Q$-complete $Q$-consistent extension of $M$. Let $A$ be a well-formed formula: when $QQA$ is a theorem of $N$ we put $v_N(A) = 1$; when $QA$ is a theorem of $N$ we put $v_N(A) = 0$. As $N$ is both $Q$-complete and $Q$-consistent, this defines a function $v_N : {\it wf} \rightarrow \{ 0, 1 \}.$

\medbreak 

{\bf Claim:} $v = v_N$ is a valuation: that is, $v(A \supset B) = 0$ precisely when $v(A) = 1$ and $v(B) = 0$. 

\medbreak 

{\it Proof}: Suppose that $v(A) = 0$: that is, suppose $QA \in \mathbb{T} (N)$; Theorem \ref{Robbin} part (7) tells us that 
$$QA \supset QQ(A \supset B) \in \mathbb{T} (L) \subseteq \mathbb{T} (N)$$
whence modus ponens places $QQ(A \supset B)$ in $\mathbb{T} (N)$ and $v(A \supset B) = 1$. Suppose $v(B) = 1$: that is, suppose $QQB \in \mathbb{T} (N)$; Theorem \ref{Robbin} part (5) tells us that 
$$QQB \supset QQ(A \supset B) \in \mathbb{T} (L) \subseteq \mathbb{T} (N)$$
whence modus ponens places $QQ(A \supset B)$ in $\mathbb{T} (N)$ and $v(A \supset B) = 1$. Thus 
$$(v(A) = 0) \vee (v(B) = 1) \Rightarrow v(A \supset B) = 1$$
and so 
$$v(A \supset B) = 0 \Rightarrow (v(A) = 1) \wedge (v(B) = 0).$$ 
Conversely, let $v(A) = 1$ and $v(B) = 0$: thus, $QQA$ and $QB$ are theorems of $N$; part (6) of Theorem \ref{Robbin} tells us that 
$$QQA \supset [QB \supset Q(A \supset B)] \in \mathbb{T} (L) \subseteq \mathbb{T} (N)$$
whence two applications of modus ponens yield $Q(A \supset B) \in \mathbb{T} (N)$ and so $v(A \supset B) = 0$. 
\begin{flushright} 
$\Box$
\end{flushright} 

We are now able to prove the completeness of IPC. 

\begin{theorem} 
The Implicational Propositional Calculus is complete. 
\end{theorem} 

\begin{proof} 
Suppose that $Q$ is not a theorem of $L$; thus, $L$ is $Q$-consistent. Theorem \ref{completion} fashions a $Q$-complete $Q$-consistent extension $N$ of $L$ by means of which we define the valuation $v_N$ as above. Before stating Theorem \ref{Robbin} we noted that $Q \supset Q$ is a theorem of $L$; thus $QQ \; (= Q \supset Q)$ is a theorem of $N$ and so $v_N (Q) = 0$. We have found a valuation under which $Q$ does not take the value $1$; $Q$ is not a tautology. 
\end{proof}

\bigbreak

\begin{center} 
{\small R}{\footnotesize EFERENCES}
\end{center} 
\medbreak 

[1] Alonzo Church, {\it Introduction to Mathematical Logic}, Princeton University Press (1956). 

\medbreak 

[2] Joel W. Robbin, {\it Mathematical Logic - A First Course}, W.A. Benjamin (1969); Dover Publications (2006). 

\medbreak

\end{document}